\documentclass[12pt,reqno]{amsart}
\usepackage{amsmath,amssymb,comment,anysize,centernot,enumitem}
\usepackage{faktor}

\renewcommand{\epsilon}{\varepsilon}
\renewcommand{\phi}{\varphi}

\newcommand{\su}{\subseteq}
\newcommand{\rest}{\restriction}
\renewcommand{\a}{\alpha}
\renewcommand{\b}{\beta}
\newcommand{\g}{\gamma}

\renewcommand{\d}{\delta}
\renewcommand{\l}{\lambda}
\renewcommand{\k}{\kappa}

\newcommand{\SetOf}[2]{\left\{#1 \ \left| \ #2 \right.\right\}}
\newcommand{\z}{\zeta}

\newcommand{\om}{\omega}

\newcommand{\lng}{\langle}
\newcommand{\rng}{\rangle}
\newcommand{\ov}{\overline}
\newcommand{\sm}{\setminus}

\newcommand{\rec}{\circledast}

\newcommand{\ehmovp}{\aleph_1\norrow_{\ov p}[\faktor{{\scriptstyle{{\aleph_0}\circledast\aleph_1}}}{{}^{1\circledast\aleph_1}}]^2_{\aleph_1}}

\newcommand{\ehmk}{\kappa^+\norrow[\faktor{{\scriptstyle{{\kappa}\circledast\kappa^+}}}{{}^{1\circledast\kappa^+}}]^2_{\kappa^+}}

\newcommand{\ehmkovp}{\kappa^+\norrow_{\ov p}[\faktor{{\scriptstyle{{\kappa}\circledast\kappa^+}}}{{}^{1\circledast\kappa^+}}]^2_{\kappa^+}}

 \newcommand{\suc}{\text{\bf succ}}

\newcommand{\dom}{{\operatorname {dom}\,}}
\newcommand{\cf}{{\operatorname {cf}}}

\newcommand{\otp}{{\operatorname {otp}}}
\newcommand{\ran}{{\operatorname  {ran}}}

\newcommand{\pp}{\operatorname {pp}}

\newcommand{\force}{\Vdash}
\newcommand{\forces}[2]{\Vdash_{#1} \mbox{``} #2 \mbox{''}}

\newcommand{\imply}{\Rightarrow}

\newcommand{\Q}{\mathbb Q}
\renewcommand{\Pr}{\operatorname{Pr}}
\newcommand{\pl}{\operatorname{P\ell}}

\newcommand{\Poset}{{\mathbb P}}
\DeclareMathOperator{\non}{\mathbf{non}}
\newcommand{\Name}[1]{\dot{#1}}
\renewcommand{\forces}[2]{\Vdash_{#1} \mbox{``} #2 \mbox{''}}

\newcommand{\norrow}{\nrightarrow}

\newcommand{\stick}{{\ensuremath \mspace{2mu}\mid\mspace{-12mu} 
{\raise0.6em\hbox{$\bullet$}}}}

\newcommand{\wfomom}{{\omega}^{<\omega}}

\newtheorem{theorem}{Theorem}
\newtheorem{corollary}[theorem]{Corollary}
\newtheorem{lemma}[theorem]{Lemma}

\newtheorem{definition}[theorem]{Definition}
\newtheorem{claim}[theorem]{Claim}
\newtheorem{remark}[theorem]{Remark}
\newtheorem{question}[theorem]{Question}
\newtheorem{fact}[theorem]{Fact}
\newtheorem{proposition}[theorem]{Proposition}

\title{Strong colorings over partitions}
\author{William Chen-Mertens}
\address{Department of Mathematics \& Statistics\\
 York University\\
4700 Keele Street \\
Toronto, Ontario, Canada M3J 1P3}
\email{chenwb@gmail.com}

 \author{Menachem Kojman}

 \address{Department of Mathematics\\
 Ben-Gurion University of the Negev\\
P.O.B. 653 \\
 Be'er Sheva\\
 84105 Israel}
 \email{kojman@woobling.org}
 
 \author[J. Stepr\={a}ns]{Juris Stepr\={a}ns}
	\address{Department of Mathematics \& Statistics, York University,
	4700 Keele Street,
	Toronto, Ontario, Canada M3J 1P3}
	\curraddr{}
	\email{steprans@yorku.ca}

 \thanks{The first autor's research for this paper was partially supported by an Israeli Science Foundation grant number 665/20.}

\thanks{The third author's research for this paper was partially supported by NSERC of
Canada. }

\subjclass[2000]{Primary: 	03E02, 03E17, 	03E35. Secondary: 	03E50}
 \keywords{Strong coloring, Ramsey theory, Generalized Continuum Hypothesis, Forcing, Martin Axiom}

\begin{document}
\maketitle
\begin{abstract}
  A strong coloring on a cardinal $\k$ is a function $f:[\k]^2\to \k$
  such that for every $A\su \k$ of full size $\k$, every color
  $\g<\k$ is attained by $f\rest[A]^2$.  The symbol 
  \[\k\norrow[\k]^2_\k
  \]
  asserts the existence of a strong coloring on $\k$.
  
  We introduce  the symbol
  \[\k\norrow_p[\k]^2_\k
  \]
which asserts the existence of  a coloring $f:[\k]^2\to \k$ 
  which is \emph{strong over a partition} $p:[\k]^2\to\theta$. 
  A coloring  $f$ is strong over $p$ if for every $A\in [\k]^\k$
  there is $i<\theta$ so that for every
  color $\g<\k$ is attained by $f\rest ([A]^2\cap p^{-1}(i))$.

 We prove that 
 whenever    $\k\norrow[\k]^2_\k$  holds, also  
 $\k\norrow_p[\k]^2_\k$  holds for an arbitrary finite partition $p$. 
 Similarly, arbitrary finite $p$-s can be added to stronger 
 symbols which hold in any model of ZFC. If $\k^\theta=\k$, then $\k\norrow_p[\k]^2_\k$ and stronger symbols, like $\Pr_1(\k,\k,\k,\chi)_p$ or $\Pr_0(\k,\k,\k,\aleph_0)_p$, also hold  for an arbitrary partition $p$ to $\theta$ parts. 
 
The symbols 
\[\aleph_1\norrow_p[\aleph_1]^2_{\aleph_1},\;\;\; \aleph_1\norrow_p[\aleph_1\rec \aleph_1]^2_{\aleph_1},\;\;\; \aleph_0\rec\aleph_1\norrow_p[1\rec\aleph_1]^2_{\aleph_1},
\]

\vskip-9pt
\[\Pr_1(\aleph_1,\aleph_1,\aleph_1,\aleph_0)_p,\;\;\;\text{ and } \;\;\; \Pr_0(\aleph_1,\aleph_1,\aleph_1,\aleph_0)_p
\]
hold for an arbitrary countable partition $p$ under the Continuum Hypothesis and are independent over ZFC $+$ $\neg$CH.

\end{abstract}

\section{Introduction}
The theory of strong colorings branched off Ramsey Theory in 1933
when 
Sierpinski  constructed  a coloring on $[\mathbb R]^2$ that contradicted the
uncountable generalization of Ramsey's theorem.  For many years,
pair-colorings  which keep their range even  after they are restricted
to all unordered pairs from an arbitrary, sufficiently large  set were called ``bad''; now
they are called ``strong''. 

\begin{definition} Let $\l\le \k$ be cardinals. A \emph{strong $\l$-coloring on  $\k$}
  is a function $f:[\k]^2\to \l$
such that $\l=\ran(f\rest [A]^2)$ for every $A\in [\k]^\k$. 
\end{definition}

By Ramsey's theorem there are no strong $\lambda$-dolorings on $\omega$ for $\lambda>1$. 
Sierpinski
constructed a strong $2$-coloring on the continuum and on $\aleph_1$.

  Assertions  of existence of strong colorings with various cardinal parameters
  are  conveniently phrased with    partition-calculus symbols. 
The (negative) square-brackets symbol
\[\k\norrow[\k]^2_\l,\]
  asserts the existence of a strong $\l$-coloring on $\k$. Recall that the symbol for Ramsey's theorem for pairs,
 \begin{equation}\label{ramsey}
 \om\to(\om)^2_n
 \end{equation}
 reads ``for every $f:[\om]^2\to n$ there is an infinite subset $A\su \om$ such that $f\rest [A]^2$
 is constant (omits all colors but one)". The square brackets in place of the rounded ones 
 stand for ``omits at least one color";  with the  negation on the arrow, the symbol
 $\k\norrow[\k]^2_\l$  means, then, 
 ``not for all colorings $f:[\k]^2\to \l$ at least
 one color can be omitted on $[A]^2$ for some $A\su \k$ of cardinality $|A|=\k$". 
 That is, there exists a strong $\l$-coloring on $\k$.
 
 When $2$ is replaced with some $d>0$
 the symbol states the existence of an analogous coloring of unordered $d$-tuples. As Ramsey's theorem
 holds for all finite $d>0$, strong $d$-dimensional colorings can also exist only on uncountable cardinals. 
 In what follows 
 we shall 
 address almost exlusively the case  $d=2$.

\begin{definition}
 Given a coloring $f:[\k]^2\to \l$, a set $X\su [\k]^2$ is \emph{$f$-strong}
if $\ran (f\rest X)=\ran (f)$. 
  \end{definition}
  
  The collection of $f$-strong subsets of $[\k]^2$ is clearly upwards 
closed and not necessarily closed under intersections. 

Different square-bracket  symbols require that
different families of sets are $f$-strong 
 with respect to the coloring $f$ whose existence each symbol asserts. 
 The symbol above asserts the existence of $f$ such that every \emph{$\k$-square},
that is, every  $[A]^2$ for some $A\in [\k]^\k$, is $f$-strong. 
A $(\l,\k)$-rectangle in $[\k]^2$ is a set of the form
$A\rec B= \{\{\a,\b\}: \a<\b<\k, \; \a\in A \text{ and } \b\in B\}$. Every $\k$-square contains  
a $(\mu_1,\mu_2)$-rectangle if $\mu_1\le\mu_2\le\k$;  the symbol 
\[\k\norrow [\mu_1\rec\mu_2]^2_\l,\]
which asserts the existence of $f:[\k]^2\to \l$ such that
 every $(\mu_1,\mu_2)$-rectangle $A\rec B\su [\k]^2$ is $f$-strong, is, then, stronger  than $\k\norrow[\k]^2_\l$.  
 
The next two strong-coloring symbols go beyond specifying which sets ought to be  $f$-strong. 
They require the existence of certain  patterns in the preimage of each color.  

\begin{definition}\label{symbdefs}
\begin{enumerate}
    \item 
   A coloring $f:[\k]^2\to \l$ witnesses the symbol 
\[\Pr_1(\k,\mu,\l,\chi)\] 
if for every $\xi<\chi$ and a pairwise disjoint 
 family $\mathcal A\su [\k]^{<\xi}$ of cardinality $|\mathcal A|=\mu$ and color
 $\g<\l$ there are $a,b\in \mathcal A$ with $\max a< \min b$ such that $f(\a,\b)=\g$ for
 all $\a\in a$ and $\b\in b$. The quantified  $\xi$ above is needed only in the case that $\chi\ge\cf(\kappa)$, 
 which received attention very recently. When $\chi<\cf(\k)$ we omit $\xi$ from the definition and require only 
 that $A\su [\k]^{<\chi}$ instead. 
 
 \item  A coloring $f:[\k]^2\to \l$ witnesses the symbol 
 \[\Pr_0(\k,\mu,\l,\chi)\] 
if for every $\xi<\chi$, a pairwise disjoint 
 family $\mathcal A\su [\k]^{\xi}$ of cardinality $|\mathcal A|=\mu$ and a matrix $\{\g_{i,j}:i,j<\xi\}\su \l$
 there are $a,b\in \mathcal A$ with $\max a< \min b$ such that $f(\a(i),\b(j))=\g_{i,j}$ for
 all $i,j<\xi$, where $a(i),b(j)$ are the $i^\text{th}$ and $j^\text{th}$ elements of $a$ and of $b$, respectively, in increasing order. 
 \end{enumerate}
 \end{definition}
 
 For $\chi>2$ and $\mu\ge\aleph_0$, $\Pr_1(\k,\mu,\l,\chi)$ implies
 $\k\norrow[\mu]^2_\l$ (see \ref{strength} below).
 If $\chi< \cf(\mu)$ then $\Pr_0(\k,\mu,\l,\chi)$
 implies $\Pr_1(\k,\mu,\l,\chi)$.

 Let us conclude the introduction with the remark that some authors use the term ``strong coloring" only for colorings which witness $\Pr_1$ or a stronger symbol.

 \section{A brief history of strong colorings}\label{history}
 
Strong $\k$-colorings on various cardinals $\k$ were constructed
by Erd\H os,   Hajn\'al, Milner and Rado in  the 1950's and 1960's from instances
of the GCH. For every cardinal $\k$ they were able to construct from $2^\k=\k^+$ colorings
$f:[\k^+]^2\to \k^+$ which witnessed
\[\k^+\norrow [\k\rec \k^+]^2_{\k^+},\]
and even colorings  which witnesses the stronger  
\[\ehmk\]
whose meaning is that inside every $(\k,\k^+)$-rectangle $A\rec B\su \k^+$ 
there is a $(1,\k^+)$-rectangle $\{\a\}\rec B\su A\rec B$ 
such that $\ran (f\rest (\{\a\}\rec B))=\k^+$ (see Section 
49 in \cite{ehmr}).  A coloring $f[\k^+]^2\to \k^+$ witnesses this symbol if and only if for every $B\in [\k^+]^{\k^+}$, for all but 
fewer than $\k$ ordinals $\a<\k^+$ the full range $\k^+$ is attained by $f$
on the set $\{\a\}\rec B=\{\{\a,\b\}:\a<\b\in B\}$.

Galvin \cite{galvin}, who was motivated by the problem of productivity of chain 
conditions and by earlier work of Laver, used $2^\k=\k^+$ to obtain a new  class
of  $2$-colorings, which in modern notation witness $\Pr_1(\k^+,\k^+,2,\aleph_0)$,  and
 used these colorings for constructing counter examples to the productivity of
the $\k^+$-chain condition. A straightforward modification of 
Galvin's proof actually 
gives $\Pr_1(\k^+,\k^+,\k^+,\aleph_0)$ on all successor cardinals
from $2^\k=\k^+$.

A remarkable breakthrough in the theory of strong colorings was the invention  of the method of ordinal-walks by
Todor{\v{c}}evi{\'c} \cite{stevo} (or, as
it was originally called, minimal walks).  Todor{\v{c}}evi{\'c} applied his method
to construct
strong colorings  on all successors of regulars in
ZFC with no additional axioms. With the same method Todor{\v{c}}evi{\'c} \cite{stevocube}
got in
ZFC the square bracket symbol for triples
$\om_2\norrow[\om_1]^3_\om$ and proved that
$\om_2\norrow[\om_1]^2_{\om_1}$ is equivalent to the negation of the
$(\aleph_2,\aleph_1)$ Chang conjecture. The rectangular symbol
$\k^+\norrow[\k^+\rec\k^+]^2_{\k^+}$ has been obtained since in ZFC on all succesors of uncountable regular
cardinals $\kappa$ by Shelah via further developments of
ordinal-walks. Moore \cite{moore} developed  ordinal-walks  further and provided 
the missing $\k^+=\aleph_1$ case.
Rinot and Todor{\v{c}}evi{\'c} \cite{rinotStevo} present a unified proof of the
rectangle version for all successors of regulars with a  completely
arithmetic  oscillation function.

 Shelah, following Galvin \cite{galvin}, phrased
   the  strong coloring relations  $\Pr_1(\k,\mu,\l,\chi)$ and $\Pr_0(\k,\mu,\l,\chi)$
   (and a few more!) and     proved  
             $\Pr_1(\k^{++},\k^{++},\k^{++},\k)$  for every regular $\k$ in ZFC \cite{Sh:572}. 
             Shelah also proved a criterion for stepping up from $\Pr_1$ to $\Pr_0$:
             if $\Pr_1(\k,\k,\l,\chi)$ holds, $\l=\l^{<\chi}$ and there is some "interpolant" cardinal $\rho$ such that
             $\rho^{<\chi}\le \l$, $2^\rho\ge\k$ and $\cf(\k)>\rho^{<\chi}$, then  $\Pr_0(\k,\k,\l,\chi)$ holds
              (Lemma 4.5(3), p. 170 of  
             \cite{CA}). In particular, chosing $\rho=\l$ as the interpolant,      $\Pr_1(\l^+,\l^+,\l^+,\aleph_0)\imply \Pr_0(\l^+,\l^+,\l^+,\aleph_0)$ for every cardinal $\l$; so for all  regular cardinals $\k$,  $\Pr_0(\k^{++},\k^{++},\k^{++},\aleph_0)$ holds in ZFC  ($\Pr_1(\aleph_1,\aleph_1,3,\aleph_0)$ cannot hold in ZFC because under MA the
            product of two ccc spaces is ccc).
       See the survey  in \cite{rinot18} for more
            background on strong colorings and non-productivity of  chain conditions.

On successors of singulars, Todor{\v{c}}evi{\'c} \cite{stevo} proved that the pcf assumption $\pp(\mu)=\mu^+$
for a singular $\mu$ implies $\mu^+\norrow[\mu^+]^2_{\mu^+}$. Shelah proved 
$\Pr_1(\mu^+,\mu^+,\cf(\mu),\cf(\mu))$ for every singular $\mu$ (4.1 p. 67 of \cite{CA}). 
Eisworth \cite{eisMC1} proved $\Pr_1(\mu^+\mu^+,\mu^+,\cf(\mu))$ from $\pp(\mu)=\mu^+$. Then Rinot, building on Eisworth's \cite{eisMC1,eisMC2},  proved that for every singular $\mu$, $\Pr_1(\mu^+,\mu^+,\mu^+,\cf(\mu))$ holds iff 
$\mu^+\norrow[\mu^+]^2_{\mu^+}$ holds. In particular, via Shelah's criterion, $\Pr_0(\mu^+,\mu^+,\mu^+,\aleph_0)\iff \mu^+\norrow[\mu^+]^2_{\mu^+}$ for all singular $\mu$  \cite{rinotrts}. Quite recently Peng and Wu proved in \cite{PW18} that 
 $\Pr_0(\aleph_1,\aleph_1,\aleph_1,n)$ holds for all $n<\omega$ outright in ZFC.

 The most recent progress on strong colorings is made in a series of papers by Rinot and his collaborators.
The result in \cite{rinot15}, shown to be optimal in Theorem~3.4 in \cite{rinotChris}, establishes the property $\Pr_1(\lambda,\lambda,\lambda,\chi)$
for regular $\lambda>\chi^+$  from a non-reflecting stationary subset of $\lambda$ composed of ordinals of
cofinality $\ge\chi$ (using a new oscillation function called $\pl_6$). In \cite{rinot18}, Rinot gets the
same result from $\square(\lambda)$, thus establishing that if $\lambda=\cf(\lambda)>\aleph_1$ and the $\lambda$-chain condition
is productive, then $\lambda$ is weakly compact in $L$. Then   Rinot and Zhang prove in \cite{rinotZhang} that 
for every regular cardinal $\kappa$, $2^\kappa=\kappa^+$ implies $\Pr_1(\kappa^+,\kappa^+,\kappa^+,\kappa)$
and for every inaccessible $\lambda$ such that $\square(\lambda)$ and $\diamondsuit^*(\lambda)$ both hold,
$\Pr_1(\lambda,\lambda,\lambda,\lambda)$ holds as well (this is the case in which our remark about $\xi$ at 
the end of (1) of Definition \ref{symbdefs} is relevant).
In the other direction  it is proved in \cite{rinotZhang} that
 $\Pr_1(\kappa^+,\kappa^+,2,\kappa)$ fails for every singular cardinal $\kappa$
and that  $\Pr_1(\kappa^+,\kappa^+,2,\cf(\kappa)^+)$ fails for a singular limit $\kappa$ of strongly compact cardinals.

Ramsey's theorem prohibits the existence of strong colorings (with more than one color) on countable sets
for which all infinite subsets are strong, but in  topological partition theory, strong colorings may exist also on countable spaces. 
Baumgartner \cite{baumgartner}, following some unpublished
work by Galvin, constructed a coloring  $c:[\mathbb Q]^2\to \omega$ which attains all colors on every homeomorphic
copy of $\mathbb Q$. Todor\v cevi\' c \cite{stevosc} obtained the rectangular version of Baumgartner's result and
very recently, Raghavan and Todor\v cevi\' c \cite{ragtod} proved that if a Woodin cardinal exists then
for every natural 
number $k >2$, for every coloring $c:[\mathbb R]^2\to k$ there is  homeomorphic copy of $\mathbb Q$ in
$\mathbb R$ on which at most $2$ colors occur,  confirming thus a conjecture of Galvin from the 1970s. They also
proved that any regular topological space of cardinality $\aleph_n$ admits a coloring of $(n+2)$-tuples which 
attains all $\omega$ colors on every subspace which is homeomorphic to $\mathbb Q$.

\section{Strong-coloring symbols over partitions}
    
We introduce now the main new notion of  symbols with an additional parameter $p$, where $p$ is 
a partition of unordered pairs. Suppose $p:[\k]^2\to \theta$ is a partition of unordered pairs
from $\k$.
A preliminary definition of  the square brackets symbol $\k\norrow_p[\k]^2_\k$ with 
 parameter $p$  has been mentioned in the abstract: there exists a coloring 
$f:[\k]^2\to \k$ such that for every 
$A\in [\k]^\k$ there is some $p$-cell $i<\theta$ such that 
for all $\g<\k$ there is $\{\a,\b\}\in [A]^2$ such that $p(\a,\b)=i$ and $f(\a,\b)=\g$. 

However,  for  $\Pr_1$ or for $\Pr_0$ it is not possible to   require a
prescribed pattern on $a\rec b$ in both $f$ and $p$ when $a,b$ belong to an arbitrary  
$\mathcal A$,  as all such $a\rec b$ might meet more than one $p$-cell. 
What we do, then, is replace this definition by a different one. The new definition   is  equivalent to the 
initial definition in all  square-bracket symbols by Fact \ref{eqdefs} below, 
and  works
for $\Pr_1$ and $\Pr_0$.

\begin{definition}
Suppose $f:[\k]^d\to \l$ is a coloring and $p:[\k]^d\to \theta$ is a
partition for a cardinal $\k$ and natural $d>0$. Then:
\begin{enumerate}
    \item 
     For a function $\z:\theta\to \l$ and  $\ov\a\in
  [\k]^d$  we say that $f$
  \emph{ hits  $\z$ over $p$  at  $\ov \a$}, if  $f(\ov\a)=\zeta(p(\ov \a))$.

 \item 
 A set $X\su [\k]^d$ is \emph{$(f,p)$-strong} if for every $\z\in
 \l^\theta$ there is $\ov\a\in X$ such that 
 $f$ 
 hits $\z$ over $p$ at $\ov
 \a$. 
 \end{enumerate}
 \end{definition}

Thus, the initial definition of an $(f,p)$-strong $X\su [\k]^d$ ---
that  
$(X\cap p^{-1}(i))$  is $f$-strong for some \emph{fixed} $p$-cell $i$ ---
is  replaced in (2) above with the requirement  that 
 every
\emph{assigment  of colors to $p$-cells} $\z:\theta\to \l$   is hit 
by some $\ov d\in X$. The advantage of the new definition is that an 
assignment $\zeta$ can be hit in any $p$-cell, so  defining $\Pr_1$ 
and $\Pr_0$ over a partition will now make sense.


Topologically, a set $X\su [\k]^d$ is $(f,p)$-strong iff the collection  
$\{u_{\lng p(\ov \a),f(\ov \a)\rng}:\ov\a\in X\}$ is an open  cover of 
the space $\l^\theta$ of all $\theta$-sequences over $\l$
with the product topology,  where $u_{\lng i,\g\rng}$ is the basic open
set $\{\zeta\in \l^\theta: \z(i)=\g\}$.

The definitions of the main  symbols over partitions which we shall work
with are in Definition \ref{mainpartrels} below; an impatient reader 
can proceed there directly. We precede this definition with 
two useful facts about $(f,p)$-strong sets.

If $X\su [\k]^d$ is $(f,p)$-strong then for every $\g<\l$ there is
 $\ov\a\in X$ such that $f(\ov\a)=\g$ since  if $\z$ is the constant
 sequence with value $\g$ and $\ov\a\in X$ is 
   such that $f(\ov\a)=\z(p(\ov\a))$ then $f(\ov\a)=\g$.
This also follows from the next fact:

\begin{fact}\label{eqdefs}
A set  $X\su [\k]^d$ is $(f,p)$-strong if and only if  there
is some $i<\theta$ such that $\l=\ran (f\rest (X\cap p^{-1}(i)))$.
\end{fact}

\begin{proof}Suppose first that that $i<\theta$ is fixed so that
  $\l=\ran (f\rest (X\cap p^{-1}(i)))$. Let $\z\in \l^\theta$ be
  arbitrary and let $\g=\z(i)$. Fix some $\ov \a\in X$ such that
  $f(\ov \a)=\g$ and $p(\ov\a)=i$. Now $f(\ov\a)=\z(p(\ov \a))$ as
  required. 

  For the other direction suppose to the contrary that for every
  $i<\l$ there is some $\z(i)\in (\l\sm\ran(f\rest X))$. Since $X$ is
  $(f,p)$-strong, find $\ov \a\in X$ such that $f$ hits $\z$ over $p$
  at $\ov \a$. Let $i=p(\ov \a)$. Now $f(\ov\a)=\z(i)\notin \ran
  (f\rest X)$ --- a contradiction.
\end{proof}

Suppose that $h:[\k]^d\to \l^{<\mu}$ is some function into
sequences of length $<\mu$. For every partition $p:[\k]^d\to \theta$ for 
some $\theta<\mu$, let $h_p:[\k]^d\to \l\cup \{*\}$ be defined by
\[h_p(\ov \a)=
\begin{cases}
  h(\ov\a)(p(\ov\a)) & \text{ if } p(\ov\a)\in \dom(h(\ov\a))\cr
  * & \text{ otherwise}
\end{cases}
 \]
 
 Then for every $\ov\a\in [\k]^d$, if $h_p(\ov\a)\not=*$ then $h_p$ hits
 $h(\ov\a)$ over $p$
at $\ov \a$. In particular, every $X\su [\k]^d$ which is $h$-strong is also $(h_p,p)$-strong for every partition  $p$ of $[\k]^d$ to $\theta<\mu$ cells.
A simple book-keeping argument can waive the dependence of $h_p$ on $p$ 
for a set of $\le \l^{<\mu}$ partitions:

\begin{lemma}\label{singlef}
Suppose $h:[\k]^d\to \l^{<\mu}$ is given and $\ov p=\lng p_\d:\d<\l^{<\mu}\rng$ is a sequence such that $p_\d:[\k]^d\to \theta_\d$ and 
$\theta_\d<\mu$ for all $\d<\l^{<\mu}$. Then there is a single coloring 
$f:[\k]^d\to \l$ such that
for all $X\su [\k]^d$, if $X$ is $h$-strong
then $X$ is $(f,p_\d)$-strong for all $\d<\l^{<\mu}$.
\end{lemma}

\begin{proof}
Suppose $h:[\k]^d\to \l^{<\mu}$ 
 and $\ov p=\lng p_\d:\d<\l^{<\mu}\rng$ are given,
where $p_\d:[\k]^d\to \theta_\d$ and  
$\theta_\d<\mu$ for every $\d<\l^{<\mu}$.  

 Let $R=\bigcup\bigl\{\{\d\}\times \l^{\theta_\d}:\d<\lambda^{<\mu}\bigr\}$.
 As $|R|=\lambda^{<\mu}$, we may fix  a bijection 
 $t:\l^{<\mu}\to R$
 and let $g=t\circ h$. So $g:[\k]^d\to R$ and 
  every $X\su [\k]^d$ is $g$-strong iff it is $h$-strong.

 Define $f:[\k]^d\to \lambda$ by 
 \[f(\ov\a)= \z(i) \text{ if } g(\ov\a)=\lng \d,\z\rng \text{ and } p_\d(\ov\a)=i.
 \]

Let $X\su [\k]^d$ be given and assume that $X$ is  $h$-strong. 
Let $\d<\l^{<\mu}$ and some 
 desirable $\z\in \l^{\theta_\d}$ be given.
 As $X$ is $h$-strong, it is also  $g$-strong, so 
 fix  ${\ov \a}\in X$ such that $g(\ov \a)=\lng \d,\z\rng$.
 Now it holds by the definition 
 of $f$ that 
 $f(\ov\a)=\z(p_\d(\ov\a))$, that is $f$ hits $\z$ over $p_\d$ at $\ov\a\in X$. 
 \end{proof}

We define  now the main    symbols  over a partition.  We state only the case for pairs. The definitions of the square-bracket symbols for  $d\not=2$ are similar.  

\begin{definition}\label{mainpartrels}
Suppose $p:[\k]^2\to \theta$ is a partition of all
unordered pairs from a cardinal $\k$. 
\begin{enumerate}
\item The symbol 
\[\k\norrow_p[\mu]^2_{\l}
\]
asserts the existence of a coloring
  $f:[\k]^2\to \l$ such that for all $A\in [\k]^\mu$, for every $\z\in
  \l^\theta$ there is $\{ \a,\b\}\in [A]^2$ such that $f(\a,
  \b)=\z(p(\a,\b))$.
  \item The symbol \[\k\norrow_p[\mu_1\rec\mu_2]^2_{\l}
  \]
asserts the existence of
    a coloring $f:[\k]^2\to \l$ such that for all $A\in [\k]^{\mu_1}$
    and $B\in [\k]^{\mu_2}$, for every $\z\in \l^\theta$ there is
    $\{\a,\b\}\in A\rec B$ such that $f( \a,\b)=\z(p(\a,\b))$.

     \item The symbol 
     $$\Pr_1(\k,\mu,\l,\chi)_p$$
asserts the existence of a
       coloring $f:[\k]^2\to \l$ such that for every $\xi<\chi$ and a family $\mathcal
       A\su [\k]^{<\xi}$ of pairwise disjoint nonempty subsets of
       $\k$ such that $|\mathcal A|=\mu$, for every $\z\in \l^\theta$
       there are $a,b\in \mathcal A$ such that $\max a<\min b$ and
       $f(\a,\b)=\z(p(\a,\b))$ for all $\{\a,\b\}\in a\rec b$. 
       
       \item The symbol
       \[\Pr_0(\k,\mu,\l,\chi)_p\] 
 asserts the existence of a coloring $f:[\k]^2\to \l$ such that for every $\xi<\chi$, a pairwise disjoint 
 family $\mathcal A\su [\k]^{\xi}$ of cardinality $|\mathcal A|=\mu$ and a matrix $\{\z_{i,j}:i,j<\xi\}\su \l^\theta$
 there are $a,b\in \mathcal A$ with $\max a< \min b$ such that $f(a(i),b(j))=\z_{i,j}(p(a(i),b(j))$ for
 all $i,j<\xi$, where $a(i),b(j)$ are the $i^\text{th}$ and $j^\text{th}$ elements of $a$ and of $b$, respectively, in increasing order. If $\chi< \cf(\mu)$ then $\Pr_0(\k,\mu,\l,\chi)_p$
 implies $\Pr_1(\k,\mu,\l,\chi)_p$.

       \item Suppose $\ov p=\lng p_\d:\d<\d(*)\rng$ is a \emph{sequence of partitions}
       $p_\d:[\k]^2\to \theta_\d$. In each of the four symbols above, writing 
       $\ov p$ instead of  $p$ means there exists a single coloring which witnesses simultaneously
       the relation with 
       $p_\d$ in place of $p$ for each $\delta<\d(*)$.
\end{enumerate}
\end{definition}

By Fact \ref{eqdefs},  the first two symbols are equivalently defined
by requiring that for every $X\su [\k]^2$ which is a $\mu$-square or is a $(\mu_1,\mu_2)$-rectangle there is a single cell $i<\theta$ (which depends on $X$) such that $X\cap p^{-1}(i)$ is $f$-strong.

\begin{fact}\label{strength}
Suppose $\k\ge \mu\ge \l$ are cardinals. Then 
every coloring $f$ which witnesses $\Pr_1(\k,\mu,\l,3)_p$ witnesses also $\k\norrow_p[\mu\rec \mu]^2_\l$. In particular, $$\Pr_1(\k,\mu,\l,3)_p \,\,\, \imply \,\,\, \k\norrow_p[\mu\rec \mu]^2_\l$$
for every partition $p:[\k]^2\to \theta$.
\end{fact}

\begin{proof}
Fix $f:[\k]^2\to \l$ which
witnesses $\Pr_1(\k,\mu,\l,3)_p$. Let $A\rec B\su [\k]^2$ be an
arbitrary $(\mu,\mu)$-rectangle. Find inductively a
pair-wise disjoint $\mathcal A=\{a_i:i<\mu\}\su A\rec B$.  Given some 
$\z\in \l^\theta$, fix $a=\{\a,\b\}$ and
$b=\{\g,\d\}$ from $\mathcal A$ such that $\a<\b<\g<\d$ and such that
$f$ hits $\z$ over $p$ at all (four) elements $\{x,y\}\in a\rec b$. In
particular, $f$ hits $\z$ over $p$ at $\{\a,\d\}$ which belongs to
$A\rec B$.
\end{proof}

 The next
 lemma is the main tool for adding a partition parameter  to a strong-coloring symbol. 
 
\begin{lemma}\label{implications}
Suppose $\k\ge\mu\ge\l\ge \rho$ are cardinals. 
Then for every sequence of partitions $\ov p=\lng p_\d:\d<\l^{<\rho}\rng$ in which 
$p_\d:[\k]^2\to \theta_\d$ and  $\theta_\d<\rho$ for  $\d<\l^{<\rho}$:
\begin{enumerate}
\item $$\k\norrow [\mu]^2_{\l^{<\rho}} \;\;\;\imply\;\;\; \k\norrow_{\ov p}[\mu]^2_\l.$$
\item For all $\mu'\le\mu$, $$\k\norrow [\mu'\rec \mu]^2_{\l^{<\rho}}\;\;\;\imply\;\;\; \k\norrow_{\ov p} [\mu'\rec \mu]^2_\l.$$
\item 
$$\kappa^+\norrow[\faktor{{\scriptstyle{{\kappa}\circledast\kappa^+}}}{{}^{1\circledast\kappa^+}}]^2_{\lambda^{<\rho}} \;\;\;\imply \;\;\; \kappa^+\norrow_{\ov p}[\faktor{{\scriptstyle{{\kappa}\circledast\kappa^+}}}{{}^{1\circledast\kappa^+}}]^2_{\lambda}$$
\item For all $\chi>0$, $$\Pr_0(\k,\mu,\l^{<\rho},\chi)\;\;\;\imply\;\;\; \Pr_0(\k,\mu,\l,\chi)_{\ov p}.$$
\item For all $\chi>0$, $$\Pr_1(\k,\mu,\l^{<\rho},\chi)\;\;\;\imply\;\;\; \Pr_1(\k,\mu,\l,\chi)_{\ov p}.$$ 
\end{enumerate}
  \end{lemma}
  
\begin{proof}
Given any  of the first three symbols in the hypotheses  above, 
fix  a coloring $h:[\k]^2\to \l^{<\rho}$ which 
witnesses it.   
Suppose $\ov p=\lng p_\d:\d<\l^{<\rho}\rng$ is given, where $p_\d:[\k]^2\to \theta_\d$ and  
$\theta_\d<\rho$ for every $\d<\l^{<\rho}$.  

By Lemma \ref{singlef} fix $f:[\k]^2\to \l$ such that every $X\su [\k]^2$
which is $h$-strong is also $(f,p_\d)$-strong for all $\d<\l^{<\rho}$. Let $\d<\l^{<\rho}$ be arbitrary.  
Suppose that $X\su [\k]^2$ is some $\mu$-square $[A]^2$ or $X$ is some   
 $(\mu',\mu)$-rectangle $A\rec B$. Then $X$ is $(f,p_\d)$-strong.
 This proves the first two implications. For the third, let $A\rec B$ be some $(\k,\k^+)$-rectangle. By the hypothesis, there is some $\a\in A$ such that $\{\a\}\rec B$ is $h$-strong, hence it is also $(f,p_\d)$-strong. 
 
 To prove the fourth implication, let, as in the proof of 
 Lemma \ref{singlef},
 $R=\bigcup\bigl\{\{\d\}\times \l^{\theta_\d}:\d<\lambda^{<\rho}\bigr\}$, let $g:[\k]^2\to R$ witness $\Pr_0(\k,\mu,\l^{<\rho},\chi)$ and let $f(\a,\b)=\z(p_\d(\a,\b))$ when $g(\a,\b)=\lng \d,\z\rng$.  Suppose $\xi<\chi$ and $\mathcal A\su [\k]^{\xi}$ is pair-wise disjoint
 and $|\mathcal A|=\mu$. Given any $\d<\l^{<\rho}$ and
 $\{\z_{i,j}:i,j<\xi\}\su \l^{\theta_\d}$, use the fact $g$ witnesses  $\Pr_0(\k,\mu,\l^{<\rho},\chi)$ to 
 fix $a,b\in \mathcal A$ such that $\max a<\min b$ 
 and $f(\a(i),\b(j))=\lng \d,\z_{i,j}\rng$ for all $i,j<\xi$, where $a(i)$ and $b(j)$ are the $i^\text{th}$ and $j^\text{th}$ members of $a$ and of $b$ respectively. Now $f(a(i),b(j))=\z_{i,j}(p_\d(a(i),b(j))$ as required.

 The proof of the last implication is  gotten from the fourth by using constant $\z_{i,j}=\z$. 
 \end{proof}

\section{Valid symbols over partitions in ZFC and in ZFC with additional axioms}

\begin{question}
    Suppose $\k\ge \rho$   are cardinals. Which  strong-coloring symbols in   $\k$ hold over \emph{all} $<\rho$ partitions?
    \end{question}

Clearly, every  coloring which witnesses a strong-coloring symbol $\Phi$ over some 
partition $p$, witnesses the  symbol gotten by deleting $p$ from $\Phi$. The question of existence 
of strong colorings over partition therefore refines the question of  
existence of strong cvolorings in the classical sense.

   Let us mention two obvious constraints
    on obtaining strong-coloring symbols over partitions. 
   Given \emph{any} coloring $f:[\k]^2\to \l$ with $\l\ge 2$, let
    us define, for $\a<\b<\k$,  $p(\a,\b)=0\iff f(\a,\b)=0$ and $p(\a,\b)=1$
    otherwise. Then $f$ does \emph{not} witness $\k\norrow_p [\k]^2_\l$. Hence:
    \begin{fact}\label{nosingle}
    No single coloring witnesses 
    $\k\norrow_p [\k]^2_\l$ for \emph{all} $2$-partitions $p$ if $\l>1$. 
\end{fact}

   If $\theta\ge \cf(\k)$ then  there is a partition  $p:[\k]^2\to\theta$ 
    with  $|p^{-1}(i)|<\k$ for every $i<\theta$,  
   so $\k\norrow_p[\k]^2_\k$ cannot hold. This narrows down 
   the discussion  of $\k\norrow_p[\k]^2_\k$ to 
   partitions $p:[\k^2]\to \theta$ with $\theta<\cf(\k)$.

 \subsection{Symbols which are valid in ZFC}

 Every infinite cardinal $\l$ satisfies $\l^{<\aleph_0}=\l$. Therefore,
 by Lemma \ref{implications}, 
 every  symbol with $\l\ge \aleph_0$ colors which holds in ZFC
 continues to hold in ZFC  over any sequence of length  $\l$
 of finite partitions. 
 
 Let us state ZFC symbols over partitions whose classical counterparts were mentioned in Section \ref{history} above:

\begin{theorem}\label{firstZFC}
  For every regular  cardinal $\k$ and a sequence of length $\k^+$ of finite partitions of $[\k^+]^2$, 
 \[\k^+\norrow_{\ov p}[\k^+\rec \k^+]^2_{\k^+}.\]
 \end{theorem}

\begin{proof}
The symbol without $\ov p$ holds by the results of Todor{\v{c}}evi{\'c}, Moore and Shelah. Now apply Lemma \ref{implications}(1).
 \end{proof}
 
 In particular, 
 \begin{corollary}
 For every finite partition $p:[\om_1]^2\to n$, 
 \[\om_1\norrow_p [\om_1\rec \om_1]^2_{\om_1}   \;\;\; \text{ and } \;\;\;   \om_1\norrow_p[\om_1]^2_{\om_1}.\]
 \end{corollary}

\begin{theorem}
   For every sequence of length $\om_2$ of finite partitions of $[\om_2]^3$, \[\om_2\norrow_{\ov p}[\om_1]^3_\om,\] 
and $\om_2\norrow_{\ov p}[\om_1]^3_{\om_1}$ is equivalent to the negation of the
$(\aleph_2,\aleph_1)$ Chang conjecture.
\end{theorem}

\begin{proof}
The symbol  $\om_2\norrow[\om_1]^3_\om$ holds by Todorevic's \cite{stevocube},  
and now apply Lemma \ref{singlef} as in the proof of Lemma  \ref{implications}.
\end{proof}

\begin{theorem}
For every cardinal $\k$ and a list $\ov p$ of length $\k^{++}$ of finite partitions of $[\k^{++}]^2$,
 \[\Pr_1(\k^{++},\k^{++},\k^{++},\k)_{\ov p}.\]
   \end{theorem} 

\begin{proof}
By Shelah's \cite{Sh:572} and Lemma \ref{implications}(4).
  \end{proof}
  
 \begin{theorem}
For every cardinal $\k$ and a list $\ov p$ of length $\k^{++}$ of finite partitions of $[\k^{++}]^2$,
 \[\Pr_0(\k^{++},\k^{++},\k^{++},\aleph_0)_{\ov p}.\]
     \end{theorem}
 \begin{proof}
By Shelah's \cite{Sh:572},   4.5(3) p. 170 in \cite{CA} and Lemma \ref{implications}(5).
  \end{proof}

 \begin{theorem}
 For every singular cardinal $\mu$ and a sequence 
 of length $\cf(\mu)$ of finite 
partitions of $[\mu^+]^2$,
\[\Pr_1(\mu^+,\mu^+,\cf(\mu),\cf(\mu))_{\ov p}.\]
\end{theorem}

\begin{proof}
By Shelah's 4.1 p. 67 of \cite{CA}
and Lemma \ref{implications}(4). 
\end{proof}

 \begin{theorem}\label{lastZFC}
   For every singular $\mu$ and a sequence $\ov p$ of length $\mu^+$
of finite 
partitions of $[\mu^+]^2$,
\[\mu^+\norrow[\mu^+]^2_{\mu^+}\imply \Pr_1(\mu^+,\mu^+,\mu^+,\cf(\mu))_{\ov p} \wedge \Pr_0(\mu^+,\mu^+,\mu^+,\aleph_0)_{\ov p}.
\]
 \end{theorem}
\begin{proof}
Suppose $\mu^+\norrow[\mu^+]^2_{\mu^+}$. By Rinot's \cite{rinotrts}, also $\Pr_1(\mu^+,\mu^+,\mu^+,\cf(\mu))$ holds. The first conjuct now follows by Lemma \ref{implications}(4).  To get the second conjunct
observee that  by the first conjunct we have in particular  $\Pr_1(\mu^+,\mu^+,\mu^+,\aleph_0)$.   The second conjunct follows now by 4.5(3) in \cite{CA} and Lemma \ref{implications}(5).
 \end{proof}

 \subsection{Symbols from instances  of the GCH or of the SCH}

If the GCH holds then every regular cardinal $\l$ satisfies $\l^{<\l}=\l$. Thus,
\begin{theorem}[GCH]
  In  Theorems (\ref{firstZFC})--(\ref{lastZFC}) above, "finite partitions" may be replaced by $<\l$-partitions.
  \end{theorem}
  
The GCH also makes Shelah's  implication 4.5(3) from \cite{CA} valid in additional cases. For example, 

\begin{theorem}[GCH]
For every regular cardinal $\k$ and a sequence $\ov p$ of length $\k^{++}$ of $\k^+$-partitions,

 \[\Pr_0(\k^{++},\k^{++},\k^{++},\k)_{\ov p}.\]
  \end{theorem}
  
  \begin{proof}
  By Shelah's \cite{Sh:572} we have $\Pr_1(\k^{++},\k^{++},\k^{++},\k)$ in ZFC. Let $\rho=\k^+$. By the GCH,
  $\rho^{<\k}=\rho$ and $(\k^{++})^{<\k}=\k^{++}$, so $\rho$ qualifies as an interpolant in 4.5(3) p. 170 in \cite{CA} and  $\Pr_0(\k^{++},\k^{++},\k^{++},\k)$ follows.  Now use GCH  again with Lemma \ref{implications}(5).
  \end{proof}

\begin{theorem}\label{ehmrr}
  For every cardinal $\k$, if $2^\k=\k^+$ then for every  sequence $p$ of length $\k^+$ of $\k$-partitions of $[\k]^2$, 
\[
\ehmkovp.
\]
\end{theorem}
\begin{proof}
The
 symbol $\ehmk$ 
 follows from $2^\k=\k^+$ by the
 and Erd\H os-Hajnal-Milner  theorem (see Section 
49 in \cite{ehmr}).  Use now Lemma \ref{implications}(2).
\end{proof}

\begin{theorem}
 For every singular cardinal $\mu$, if $\pp(\mu)=\mu^+$
then for every  sequence $\ov p$ of length $\mu^+$
of finite 
partitions of $[\mu^+]^2$,
\[ \Pr_1(\mu^+,\mu^+,\mu^+,\cf(\mu))_{\ov p}. 
\]
\end{theorem}
 
 \begin{proof}
By $\pp(\mu)=\mu^+$  and Eisworth's theorem \cite{eisMC2}, $\Pr_1(\mu^+,\mu^+,\mu^+,\cf(\mu))$ holds. Now use Lemma \ref{implications}(4). 
 \end{proof}

 \begin{theorem}[GCH]
 For every singular cardinal $\mu$ and a  sequece $\ov p$ of length $\mu^+$
of 
$\mu$-partitions of $[\mu^+]^2$,

\[ \Pr_0(\mu^+,\mu^+,\mu^+,\cf(\mu))_{\ov p}. 
\]
 \end{theorem}
 
 \begin{proof}
 By Eisworth's theorem it holds that $\Pr_1(\mu^+,\mu^+,\mu^+,\cf(\mu))$. By the GCH and Shelah's 4.5(3) in \cite{CA}, also $\Pr_0(\mu^+,\mu^+,\mu^+,\cf(\mu))$ holds. Finally, as $(\mu^+)^\mu=\mu^+$, by
 Lemma \ref{implications}(5), for every sequence $\ov p$ of length $\mu^+$ of $\mu$-partitions of $[\mu]^+$ it holds that $\Pr_0(\mu^+,\mu^+,\mu^+,\cf(\mu))_{\ov p}$. 
 \end{proof}
 
 In the next theorem a different cardinal arithmetic assumption appears:
 
 \begin{theorem}If 
$\mu$ is a singular cardinal and
 $2^{\cf(\mu)}>\mu$   then for every  sequence $\ov p$ of length $\mu^+$ of finite 
partitions of $[\mu^+]^2$, 
\[ \Pr_0(\mu^+,\mu^+,\cf(\mu),\aleph_0)_{\ov p}. 
\]
\end{theorem} 
\begin{proof}
By Shelah's 4.1 p. 67 the symbol $\Pr_1(\mu^+,\mu^+,\cf(\mu),\cf(\mu))$ holds in ZFC. Choose $\rho=\cf(\mu)$. So $2^\rho\ge \mu^+$, $\rho^{<\aleph_0}=\rho$ and $\cf(\mu^+)>\rho^{<\aleph_0}$, so $\rho$ qualifies as an interpolant cardinal in 
4.5(3) p. 170 in \cite{CA} and $\Pr_0(\mu^+,\mu^+,\cf(\mu),\aleph_0)$ follows. Now  use Lemma \ref{implications}(5).
\end{proof}

Lastly in this section,  we show that  $\stick(\k)$, an axiom (stated in the proof below),   
which does not imply $2^\k=\k^+$,  implies the following rectangular square-brackets symbol.

\begin{theorem} \label{stick}
If $\k$ is a cardinal and  $\stick(\k^+)$ holds 
then for every  sequence of partitions $\ov p=\lng p_\g:\g<\k^+\rng$,
where
$p_\g:[\k^+]^2\to \l_\g$ and $\l_\g<\cf(\k)$ for each $\g<\k^+$, it holds that
\[\ehmkovp.
\]

\noindent{That is, there exists a coloring $f:[\k^+]^2\to \k^+$  
such that for every $(\k^+,\k^+)$-rectangle $A\circledast B$ 
and $\g<\om_1$ there is $j<\l_\g$ and $X\in [A]^\k$ such that 
  such that
   \[
   \k^+=\ran \bigl(
   f\rest [(X\circledast B)\cap p_\g^{-1}(j)]
   \bigr).
   \]}
\end{theorem}

\begin{proof}
Suppose a sequence of partitions $\ov p =\lng p_\g:\g<\k^+\rng$ is
given as above and we shall define the required $f$ assuming
$\stick(\k^+)$.  Fix a sequence $\lng X_i:i<\k^+\rng$ which witnesses
$\stick(\k^+)$, that is: $X_i\su \k^+$, $\otp (X_i)=\k$ for each
$i<\k^+$ and for every $A\in [\k^+]^{\k^+}$ there exists some $i<\k^+$
such that $X_i\su A$.

Let $\beta<\k^+$ be arbitrary. Towards defining $f(\a,\b)$ for $\a<\b$,
 let us define, for every triple $\lng\g,i,j\rng$ such that
$\g,i<\b$ and $j<\l_\g$,
  \begin{equation}\label{defofA}
  A^\b_{\lng \g,i,j\rng}
  = \{\a<\b: \a\in X_i \wedge p_\g(\a,\b)=j\}
  \end{equation}

Let 

\begin{equation}\label{2defofA}
    \mathcal
A_\b=\{A^\b_{\lng \g,i,j\rng}: \g,i<\b\wedge  j<\l_\g \wedge
|A^\b_{\lng \g,i,j\rng}|=\k\}.
\end{equation}

As $\mathcal A_\b$ is a  family of at most $\k$ subsets of $\b$,
each of cardinality $\k$,
 we may fix  a disjoint refinement
$\mathcal D_\b=\{D^\beta_{\lng \g,i,j\rng}:
A^\b_{\lng \g,i,j\rng}\in \mathcal A_\b\}$, that is, 
each $D^\beta_{\lng \g,i,j\rng}\su A^\b_{\lng \g,i,j\rng}$ has cardinality 
$\k$ and $\lng \g,i,j\rng\not=\lng \g',i',j'\rng\imply
D^\b_{\lng \g,i,j \rng}\cap D^\b_{\lng \g',i',j'\rng}=\emptyset$
for any  $A^\b_{\lng \g,i,j\rng}, A^\b_{\lng \g',i',j'\rng}\in \mathcal A_\b$.

 Let us define now $f(\a,\b)$ for all $\a$ below our fixed $\b$ by cases.
 For each $D^\beta_{\lng \g,i,j\rng}\in \mathcal D_\b$ define
 $f\rest(D^\beta_{\lng \g,i,j\rng}\circledast \{\b\})$ to be some        
function \emph{onto} $\b$.  This is possible since $| D^\beta_{\lng \g,i,j\rng}|=\k$ 
and $\b<\k^+$ (so 
$| D^\beta_{\lng \g,i,j\rng}|=|\b|$) and because 
the $D^\beta_{\lng \g,i,j\rng}$ are pairwise disjoint, 
hence $(D^\beta_{\lng \g,i,j\rng} \rec \{\b\})\cap (D^\beta_{\lng \g',i',j'\rng}\rec 
\{\b\}) =\emptyset$ 
when $\lng \g,i,j\rng\not=\lng \g',i',j'\rng $.

For $\a\in \b\sm\bigcup
\mathcal D_\b$ define $f(\a,\b)$ arbitrarily (say, as $0$).  
As $\b$ was arbitrary,
we have 
defined $f(\a,\b)$ for all $\a<\b<\k^+$.
By this definition, for all $\b<\k^+ $ and $D^\b_{\lng \g,i,j\rng}\in \mathcal D_\b$,

\begin{equation} \label{ranbeta}
  \b=\ran (f\rest (D^\beta_{\lng \g,i,j\rng}\circledast \{\b\}).
\end{equation}

To see that $f$ satisfies what Theorem~\ref{stick} states, 
let $A,B\su \k^+$ be 
arbitrary  with $|A|=|B|=\k^+$ and let $\g<\k^+$ be given.
Using the properties of the  $\stick(\k^+)$-sequence, fix some $i<\k^+$ 
such that
\begin{equation}
  X_i\su A\label{XiinA}
\end{equation}
As $X_i\su \k^+$ and $\otp(X_i)=\k$, $\sup(X_i)<\k^+$,
hence $\b_0:=\max\{\g,i,\sup X_i\}<\k^+$.

  If $\beta\in B$ is any ordinal such that  $\b>\b_0$ then  $X_i\su \b$ and 
  as  $|X_i|=\k$ while $\l_\g<\cf\k$, there exists some $j(\b)<\l_\g$ such that 
  \[
  |\{\a\in X_i: p_\g(\a,\b)=j \}|=\k,\]
  that is, by~(\ref{defofA}) and~(\ref{2defofA}), 
  \[A^\b_{\lng \g,i,j(\b)\rng}\in \mathcal A_\b.
  \]
   By the regularity of $\k^+$  and the assumption that 
   $\l_\g<\cf\k<\k^+$, we can fix
  some $B'\su B\sm (\b_0+1)$ and $j(*)<\l_\g$ such that $j(\b)=j(*)$ for all $\b\in B’$.
  
  For each $\b\in B'$ it holds, then, that  $A^\b_{\lng\g,i,j(*)\rng}$
  belongs to $\mathcal A_\b$, 
  and therefore also
  \begin{equation}\label{j*beta}
  	D^\b_{\lng \g,i,j(*)\rng}\in \mathcal D_\b.
  \end{equation}
Now, for each $\b\in B'$ we have by~(\ref{j*beta}) and~(\ref{ranbeta}) that 
\[\b=\ran(c\rest D^\b_{\lng \g,i,j(*)\rng}\circledast \{\b\}),\]
and as $D^\b_{\lng \g,i,j(*)\rng}\su X_i\cap p_\g^{-1}(j(*))$
by~(\ref{defofA}),
\[\b=\ran(f\rest [(X_i\circledast \{\b\})\cap p_\g^{-1}(j(*))]\bigr).\]
As $B'\su B$ is unbounded in $\k^+$ it follows, after 
setting $X=X_i$ and $j=j(*)$, that

\[
\k^+= f\rest [(X\circledast B)\cap p_\g^{-1}(j)].
\]

\end{proof}

\section{Independence results on  $\aleph_1$.}\label{ind}
  
 In this Section we shall  show that the existence of strong
 colorings over \emph{countable} partitions of $[\om_1]^2$ is independent
 over ZFC and over ZFC + $2^{\aleph_0}>\aleph_1$.
 \begin{theorem}
If the CH holds, then the following five symbols are valid for every sequence of partitions $\ov p=\lng p_\d:\d<\om_1\rng$ where
$p_\d:[\om_1]^2\to \om$: 
\begin{itemize}
\item $\displaystyle \aleph_1\norrow_{\ov p}[\aleph_1]^2_{\aleph_1}$
\item $\displaystyle\aleph_1\norrow_{\ov p}[\aleph_1\rec \aleph_1]^2_{\aleph_1}$
\item $\displaystyle \ehmovp$
\item $\displaystyle\Pr_1(\aleph_1,\aleph_1,\aleph_1,\aleph_0)_{\ov p}$
\item $\displaystyle \Pr_0(\aleph_1,\aleph_1,\aleph_1,\aleph_0)_{\ov p}.$
\end{itemize}
\end{theorem}

\begin{proof}
Assume CH, that is,
 $2^{\aleph_0}=\aleph_1$. Then
 $\Pr_1(\aleph_1,\aleph_1,\aleph_1,\aleph_0)$ holds by (a slight strengthening of) Galvin's 
 theorem. By Shelah's 4.5(3) from \cite{CA}, also $\Pr_0(\aleph_1,\aleph_1,\aleph_1,\aleph_0)$ holds.
The CH  also implies that $(\aleph_1)^{\aleph_0}=(2^{\aleph_0})^{\aleph_0}=2^{\aleph_0}=\aleph_1$.
 By Lemma \ref{implications}(5), then, for every $\om_1$-sequence $\ov p$ of countable partitions of $[\om_1]^2$ it holds that
 \[\Pr_0(\aleph_1,\aleph_1,\aleph_1,\aleph_0)_{\ov p}\]
  and therefore by Lemma \ref{strength} also

 \[\Pr_1(\aleph_1,\aleph_1,\aleph_1,\aleph_0)_{\ov p}, \;\;\;\om_1\norrow_{\ov p}[\om_1]^2_{\om_1}
 \;\;\; \text{ and } \;\;\;  \om_1\norrow_{\ov p}[\om_1\rec\om_1]^2_{\om_1}.\]

 Similarly, by the CH and Theorem \ref{ehmrr} in the previous Section,
 \[\ehmovp.
 \]
  \end{proof}

 We  prove next  that these five
  symbols are valid  in all models of ZFC obtained by adding $\aleph_2$
 Cohen reals over an arbitrary model $V$ of ZFC, and, more generally, by forcing with a finite-support
 $\om_2$-iteration of $\sigma$-linked posets over an arbitrary  model $V$  of ZFC.

Before proving yet another combinatorial property in a Cohen extension let us recall Roitman's \cite{roitman} proof that the addition of a single Cohen real
 introduces an $S$-space, Todor{\v{c}}evi{\'c}'s presentation in \cite{stevoBirk}, p. 26 and   Rinot's blog-post
 \cite{rinotblog} in which it is shown that a single Cohen real
 introduces $\Pr_0(\aleph_1,\aleph_1,\aleph_0,\aleph_0)$. For a short proof 
 of Shelah's theorem that a single Cohen real introduces a Suslin
 line see \cite{stepransCohen}. Fleissner \cite{fleissner} proved that adding $\l$ Cohen reals introduces two ccc spaces whose product is not $\l$-cc. 
 Hajnal and Komjath \cite{HajKomSim} proved that adding one Cohen subset to a cardinal $\k=\k^{<\k}$ forces the statement $Q(\k^+)$ they defined, following \cite{egh}: for every graph $G=\lng \k^+,E\rng$ with $\chi(G)=\k^+$ there is a coloring
 $f:E\to \k^+$ such that for
every partition of $\k^+$ to $\k$ parts, all colors are gotten by $f$ on edges from 
a single part. It is still open if $Q(\aleph_1)$ holds in ZFC. 
    

  \begin{theorem}\label{TheRecCohen}
If $\mathbb C_{\aleph_2}$ is the partial order for adding $\aleph_2$ Cohen reals then  for every sequence $\ov p =\lng p_\d:\d<\om_1\rng$ of partitions $p_\d:[\om_1]^2\to \om$ in the forcing extension by $\mathbb C_{\aleph_2}$,
\[
1\forces{\mathbb C_{\aleph_2}}{\ehmovp.
}\]
\end{theorem}

\begin{proof}
Let $\mathbb{C}_{\alpha}$ be the partial order of finite partial functions from $[\alpha]^2$ to $\omega$.
Let $V$ be a model of set theory 
and let $G\subseteq \mathbb C_{\omega_2}$ be generic over $V$. 
Then $\bigcup G:[\omega_2]^2\to\omega$.


Now suppose that $\ov p=\lng p_\d:\d<\om_1\rng$ 
is an arbitrary sequence of partitions $p_\d:[\om_1]^2\to \om $ in
$V[G]$. As there is some $\alpha\in \omega_2$ such that $\ov p\in
V[G\cap \mathbb C_\alpha]$, it may be assumed that $\ov p\in V$. 
Let 
$c= \bigcup G\restriction [\omega_1]^2$. So $c:[\omega_1]^2\to \omega$. 
In $V$, fix a sequence $\langle e_\a:\omega\le \alpha<\omega_1\rangle$, where $e_\a: \om\to \a$ is a bijection. In the generic extension, define a coloring $f:[\om_1]^2\to \om_1$ by 
\[f(\a,\b)=e_\b(c(\a,\b)),\]
for $\b\ge \om$ and as $0$ otherwise. 

To see that $f$ witnesses $\omega\rec\omega_1\norrow_{\ov p}[1\rec \omega_1]^2_{\omega_1}$
suppose that  there is some $\d<\om_1$ for which $f$ fails to witness
$\omega\rec\omega_1\norrow_{ p_\d}[1\rec \omega_1]^2_{\omega_1}$. This means that in $V^G$ there are $A\in [\om_1]^\om$ and $B\in [\om_1]^{\om_1}$ such that for all $\a\in A$ there is some $W(\a)\in\om_1^\om$ such that for all $\b\in B\sm (\a+1)$ it holds that $f(\a,\b)\not=W(\a)(p_\d(\a,\b))$. Let $\dot A$ and $\dot W$ be countable names for $A$ and $W$ and let  $\dot B$ be a name for $B$. Let $r\in G$ decide $\d$ and force 
\[r\forces{}{(\forall \a\in \dot A)(\forall \b\in \dot B\sm (\a+1))\,(f(\a,\b))\not= \dot W(\a)(p_\d(\a,\b))}
\]

Let $\mathfrak M$ be a countable elementary submodel of $H(\omega_2,\dot A, \dot B, \dot{W}, r) $.

 Fix an extension $r'\in G$ of $r$ and an ordinal $\b\in \om_1\sm \sup (\mathfrak M\cap \om_1)$ such that $r'\force \b\in \dot B$. Let $r_0=r'\cap \mathfrak M$. Inside $\mathfrak M$ extend $r_0$ to $r_1$ such that $r_1\forces{}{\a\in \dot A}$ for an ordinal $\a$ which is not in $\bigcup\dom(r')$ and $r_1$ decides $W(\a)(p_\d(\a,\b))$. Thus, $\{\a,\b\}\notin \dom(r'\cup r_1)$. Let $$r^*=r'\cup r_1\cup \bigl\{\lng \{\a,\b\},e^{-1}_\b(W(\a)(p_\d(\a,\b))\rng\bigr\}.$$
 
 Since $r^*$ extends $r$ and $f(\a,\b)=e_\beta(c(\a,\b))=W(\a)(p_\d(\a,\b))$, this is a contradiction to the choice of $r$. 
 
\end{proof}


 The forcing for adding a single Cohen real is obviously $\sigma$-linked. Thus, the next theorem applies to a broader class of posets  than Cohen forcing.  The previous theorem  holds also in this generality. 
 
 \begin{theorem}\label{TheRecSigmalinked}
If $\mathbb P$ is an $\omega_2$-length finite support iteration of $\sigma$-linked partial orders then
$$1\forces{\mathbb P}{\Pr_{0}(\aleph_1,\aleph_1,\aleph_1,\aleph_0)_{\bar p}}$$
for any $\omega_1$ sequence of partitions $\bar{p} = \lng p_\d:\d<\omega_1\rng$ such that 
$p_\d:[\omega_1]^2\to \omega$ for all $\d<\om_1$.
\end{theorem}

\begin{proof}
Let $\mathbb{P}_{\alpha}$ be the finite support iteration of the first $\alpha$ partial orders
and suppose that 
\begin{equation}
1\forces{\mathbb P_\alpha}{\mathbb Q_\alpha = \bigcup_{n\in\omega} \mathbb Q_{\alpha,n}
\text{ and each $\mathbb Q_{\alpha,n}$ is linked}}
\end{equation}

\begin{equation}
1\forces{\mathbb P_\alpha}{\{\dot{q}_{\alpha,n}\}_{n\in\omega}
\text{ is a maximal antichain in }
\mathbb Q_{\alpha}}
\end{equation}
Let $B: [\omega_2]^2\to \omega_2$ be a bijection
and let $e_\xi:\omega\to \xi$ be a bijection for each infinite $\xi \in \omega_1$.
Let $V$ be a model of set theory, let $G\subseteq \mathbb P$ be generic over $V$ and 
let $G_\alpha$ be the generic filter induced on $\mathbb Q_\alpha$ by $G$.

Now suppose that a
sequence of partitions $\bar{p} = \lng p_\d:\d<\omega_1\rng$ such that 
$p_\d:[\omega_1]^2\to \omega$ belongs to $V[G]$.
 As there is some $\alpha\in \omega_2$ such that $\bar{p}\in
V[G\cap \mathbb P_\alpha]$, it may be assumed that $\bar{p}\in V$. 
There is no harm in assuming that $B$ maps $ [\omega_1]^2$  to $\omega_1$ so let 
$c: [\omega_1]^2\to \omega_1$
be defined by $c(\alpha,\beta) = \xi$ if and only if $q_{B(\alpha,\beta),k}\in G_{B(\alpha,\beta)}$ and $\xi = e_{B(\alpha,\beta)}(k)$.

To see that $c$ witnesses $\Pr_0(\aleph_1,\aleph_1,\aleph_1,\aleph_0)_{\bar{p}}$
suppose that:
\begin{itemize}
\item $\delta \in \omega_1$
\item $k>0$ 
\item $\dot \a_{\xi,i}<\om_1$ are $\mathbb P$-names for $\xi<\om_1$ and $1\le i\le k$
of distinct ordinals such that for every $\xi<\om_1$ the 
   sequence $\lng \dot\a_{\xi,i}:1\le i\le\k\rng$ is increasing with $i$
   \item 
  $\left(\dot{N}_{i,j}\right)$  is a $\mathbb P$-name for a $k\times k$ matrix with entries in ${\omega_1}^\omega$.
  \end{itemize}
We may fix $q_\xi\in G$
such that:
\begin{enumerate}
\item $q_\xi\forces{\mathbb P}{\dot{\a}_{\xi,i} = \a_{\xi,i}}$ for all $1\le i\le k$
\item  $\{\a_{\xi,1},\a_{\xi, 2}, \ldots , \a_{\xi, k}\}\subseteq d_\xi = \bigcup B^{-1}(\dom(q_\xi))$
\item \label{silinas}for all $\mu \in \dom(q_\xi)$ there is $n_{\mu,\xi}$ such that 
$q_\xi\restriction \mu\forces{\mathbb P_\mu}{q_\xi(\mu)\in \mathbb Q_{\mu, n_{\mu,\xi}}}$.
\end{enumerate}

Let  $\{\mathfrak M_\eta\}_{\eta\in \omega+1}$ be  countable elementary submodels of $$H(\omega_2,\{q_\xi, \{\a_{\xi,1},\a_{\xi, 2}, \ldots , \a_{\xi, k}\}\}_{\xi\in\omega_1},B, \dot{N}, G) $$
such that $\mathfrak M_j\prec \mathfrak M_{j+1}\prec \mathfrak M_{\omega}$ and
$\omega_1\cap \mathfrak M_j \in \mathfrak M_{j+1}$ for each $j\in \omega$.
Let $\xi_\omega \in \omega_1\setminus \mathfrak M_\omega$. 
By  elementarity there are $\xi_j\in \omega_1\cap \mathfrak M_j$ such that:
\begin{enumerate}
\item $\dom(q_{\xi_\omega})\cap \mathfrak M_j\subseteq \dom(q_{\xi_j})$
\item\label{silinas2} $ n_{\mu,\xi_j} = n_{\mu,\xi_\omega}$ for each $\mu \in \dom(q_{\xi_\omega})\cap \mathfrak M_j$.
\end{enumerate}

Note that $\{\a_{\xi_\omega,1},\a_{\xi_\omega, 2}, \ldots , \a_{\xi_\omega, k}\}\cap \mathfrak M_\omega=\varnothing$ and hence
\begin{equation}
\label{FR1}
(\forall j\in \omega)(\forall u\in k)(\forall v\in k) \ B(\a_{\xi_j,u},\a_{\xi_\omega,v})\notin \mathfrak M_\omega .
\end{equation}
Furthermore,   note that $\bigcup B^{-1}(\dom(q_{\xi_\omega}))$ is finite and so there is $J$ such that
$$\bigcup B^{-1}(\dom(q_{\xi_\omega}))\cap \mathfrak M_\omega\subseteq \mathfrak M_J . $$
From (\ref{FR1})  it follows that \begin{equation}\label{FR2}B(\a_{\xi_J,u},\a_{\xi_\omega,v})\notin \dom(q_{\xi_J})\cup \dom(q_{\xi_\omega}) .\end{equation}
From condition (\ref{silinas}) in the choice of $q_\xi$ and condition (\ref{silinas2}) in the choice of $\xi_j$, it follows that there is $q^*$ such that $q^*\leq q_{\xi_J}$ and $q^*\leq q_{\xi_\omega}$ and
$\dom(q^*)= \dom(q_{\xi_J}) \cup \dom(q_{\xi_\omega})$.

Let $\mathcal{A}\in\mathfrak{M_\omega}$ be a maximal antichain such that for every conditions $r\in \mathcal A$,
$$r\forces{\mathbb P}{M_{u,v} = \dot{N}_{u,v}(p_\delta(a_{\xi_J,u}, a_{\xi_\omega,v}))}$$ 
for some $k\times k$ matrix $\left(M_{i,j}\right)$ with entries in $\omega_1$.

By the countable chain condition, $\mathcal{A}$ is countable and hence $\mathcal{A}\subseteq \mathfrak{M_\omega}$. Let 
$r\in\mathcal{A}$ be such that $r$ is compatible with $q^*$ and let $\left(M_{i,j}\right)$ be the $k\times k$ matrix which witnesses that $r\in\mathcal{A}$. Let $q^{**}\le q^*, r$.

Note that $B(\a_{\xi_J,u},\a_{\xi_\omega,v})\notin \dom(q^{**})$ because
$\dom(q^{**})\setminus (\dom(q_{\xi_J})\cup \dom(q_{\xi_\omega}))\subseteq \mathfrak M_\omega$ and (\ref{FR1}) and (\ref{FR2}) hold.
Let
$$\hat{q}(\theta)= \begin{cases}
q^{**}(\theta) & \text{ if } \theta \notin \{B(a_{\xi_J,u}, a_{\xi_\omega,v})\}_{u,v\in k}\\
q_{\theta , e^{-1}_{B(a_{\xi_J,u}, a_{\xi_\omega,v})}(M_{u,v})} & \text{ if } \theta = B(a_{\xi_J,u}, a_{\xi_\omega,v})
\end{cases}$$
Then by the definition of $c$
$$\hat{q}\forces{\mathbb P}{c(\a_{\xi_J, u}, \a_{\xi_\omega, v}) 
= M_{u,v} = \dot{N}_{u,v}((p_\delta(a_{\xi_J,u}, a_{\xi_\omega,v}))}$$
for each $u$ and $v$ as required.
\end{proof}
\begin{corollary}
It is consistent with ${MA}_{\aleph_1}(\sigma\text{-linked})$ that 
$\Pr_{0}(\aleph_1,\aleph_1,\aleph_1,\aleph_0)_{\bar p}$ holds 
for any $\omega_1$ sequence of partitions $\bar{p} = \{p_\xi\}_{\xi\in\omega_1}$ such that 
$p_\xi:[\omega_1]^2\to \omega$.
\end{corollary}

Now we prove  that the symbol
  \[\om_1\norrow_p [\om_1]^2_{\om_1}\]
 can consistently \emph{fail} for some
 $p:[\om_1]^2\to \om$. 
 
 We actually prove more. The failure of the symbol above over a partition $p:[\om_1]^2\to \om$, symbolically written as  
  \[\om_1\rightarrow_p[\om_1]^2_{\om_1},\]
 means that for every coloring $f:[\om_1]^2\to \om_1$ there is a set $A\in [\om_1]^{\aleph_1}$ such that $f\rest([A]^2\cap p^{-1}(i))$ omits at least one color for every $i<\om$. Let us introduce the following symbol:
 \[\om_1\rightarrow_p[\om_1]^2_{\om_1\sm \om_1},\]
to say that for every coloring $f:[\om_1]^2\to \om_1$ there is a set $A\in [\om_1]^{\aleph_1}$ such that for every $i<\om$ a set of size $\aleph_1$ of colors is omitted by $f\rest([A]^2\cap p^{-1}(i))$. An even stronger failure (via breaking $\om_1$ to two disjoint equinumerous sets and identifying all colors in each part) is 
\[\om_1\rightarrow_p[\om_1]^2_2.\]

It is the consistency of the latter  symbol which we prove. Note that with the rounded-brackets symbol in  (\ref{ramsey}) from the introduction we may write this failure as:
\[\om_1\rightarrow_p(\om_1)^2_2,\]
whose meaning is that for every coloring $f:[\om_1]^2\to 2$ there is $A\in [\om_1]^{\aleph_1}$ such that for every $i<\om$ the set 
$[A]^2\cap p^{-1}(i)$ is $f$-monochromatic. Thus, while $\om_1\norrow[\om_1]^2_{\om_1}$ holds in ZFC, it is consistent that for a  suitable countable  partition $p$ the 
symbol $\om_1\norrow_p[\om_1]^2_{\om_1}$ fails pretty badly.



\begin{theorem}\label{The8}
It is consistent that $2^{\aleph_0}=\aleph_2$ and there is a partition 
$p:[\omega_1]^2 \to \omega$ such that
\[\om_1\rightarrow_p [\om_1]^2_2.\]
\end{theorem}

\begin{corollary}
It is consistent that $2^{\aleph_0}=\aleph_2$ and there is  some $p:[\om_1]^2\to \om$ 
such that 

\[\om_1\rightarrow_p[\om_1]^2_{\om_1\sm \om_1} \;\;\;\text{ and hence }\;\;\; \om_1\rightarrow_p[\om_1]^2_{\om_1}.
\]	
\end{corollary}

\begin{proof}[Proof of the theorem]

Let $\mathbb P$ be the partial order of finite partial functions from
$[\omega_1]^2 \to \omega$ ordered by inclusion. More precisely, each condition $q\in\mathbb P$ has associated to it a finite subset of $\omega_1$ which, abusing notation, will be called $\dom(q)$. Then $q$ is a function $[\dom(q)]^2\rightarrow \omega$. 


Given any partition $p:[\omega_1]^2 \to \omega$ and a colouring $c:[\omega_1]^2 \to 2$ 
define the partial order
$\mathbb Q(p,c)$ to be the set of all pairs $(h,w)$ such that 
\begin{itemize}
\item $w\in [\omega_1]^{<\aleph_0}$ 
\item $h:m\to 2$ for some $m\in \omega$ so that $m\supseteq p([w]^2)$
\item $c(\{\alpha,\beta\}) \neq h(p(\{\alpha,\beta\}))$
for each $\{\alpha,\beta\}\in [w]^2$ 
\end{itemize}
and order $\mathbb Q(p,c)$ by coordinatewise extension.  Let $V$ be a
model of set theory in which $2^{\aleph_1} = \aleph_2$ and let
$\{c_\xi\}_{\xi\in \omega_2}$ enumerate cofinally often the subsets of
hereditary cardinality less than $\aleph_2$.  If $G\subseteq \mathbb
P$ is generic over $V$, in $V[G]$ define $p_G=\bigcup G$. Then define a finite support iteration $\{\mathbb Q_\zeta\}_{\zeta \in
  \omega_2}$ such that $\mathbb Q_1=\mathbb P$ and if $c_\zeta$ is a $\mathbb Q_\zeta$-name such
that $1\forces{\mathbb Q_\zeta}{c_\zeta:[\omega_1]^2 \to 2}$ then
$\mathbb Q_{\zeta+1} = \mathbb Q_\zeta\ast \mathbb Q(p_G,c_\zeta)$.
 
It suffices to establish the following two claims.
 \begin{claim}\label{prevcl}
 For each $\zeta\in \omega_2$ greater than $1$ and $\eta\in \omega_1$ the set of $q\in\Q_{\zeta+1}$ such that
 $$q\restriction \zeta\forces{\Q_\zeta}{q(\zeta ) = (h,w)\text{ and }
   w\setminus \eta\neq \varnothing}$$ is dense in $\Q_{\zeta+1}$.
 \end{claim}
 \begin{proof}
 Given $q$ it may be assumed that there are $h$ and $w$ such that 
 $$q\restriction \zeta\forces{\Q_\zeta}{q(\zeta ) = (\check{h},\check{w})} .$$
 Let $\theta\in \omega_1$ be so large that $\theta>\max(\dom(q(0))),\max(w),\eta$.
 Let $f:w\to \omega$ be any one-to-one function so that $\ran(f)\cap \dom(h)=\emptyset$ and let
 $f_\theta:\{\{\theta,\rho\}\}_{\rho\in w}\to \omega$ be defined by
 $f_\theta(\{\theta,\rho\}) = f(\rho)$. Note that since $q(0)\cup f_\theta\in \mathbb P$ it is  possible to find $\bar{q}\leq q\restriction \zeta$ such that:
 \begin{itemize}
 \item $f_\theta\subseteq \bar{q}(0)$
 \item $\bar{q}\forces{\Q_\zeta}{c_\zeta(\{\theta,\rho\}) = \check{k}_{\rho}}$ for some family of integers
 $\{k_\rho\}_{\rho\in w}$ equal to $0$ or $1$.
 \end{itemize}
 Then let $\bar{h}\supseteq h$ be any finite function such that $\bar{h}(f(\rho))=1- k_\rho$ and let
 $\bar{w} = w\cup \{\theta\}$. Then $\bar{q}\ast (\bar{h},\bar{w})$ is the desired condition.
 \end{proof}

 \begin{claim}
 The partial order $\mathbb Q_{\omega_2}$ satisfies the ccc.
 \end{claim} 
 \begin{proof}
 By a standard argument, there is a dense subset of $\Q_{\omega_2}$ of conditions $q$ such that for each $\zeta\in\dom(q)$ with $\zeta>0$, there are $h$ and $w$ so that $q\restriction \zeta\forces{\mathbb Q_{\omega_2}}{q(\zeta ) = (\check{h},\check{w})}$. We will assume that all conditions that we work with are members of this dense subset.
 
 Let $\{ q_\xi:\xi<\omega_1\}$ be conditions in  $\Q_{\omega_2}$. By thinning out, we can assume that their domains form a $\Delta$-system with root $\{0,\zeta_0,\zeta_1,\ldots,\zeta_k\}$. We can further assume that: 
 \begin{itemize}
\item each of the sets $\{\dom(q_\xi(0)):\xi<\omega_1\}$, and $\{w_{\xi,\zeta_i}:\xi<\omega_1\}$ for each $i\le k$ form a $\Delta$-system \item The functions $q_\xi(0)$ agree on the root of the $\Delta$-system of their domains,
\item there are $h_i$, $i\le k$, so that for all $\xi<\omega$ we have $h_i=h_{\xi,\zeta_i}$
 \end{itemize}
where $q_\xi\restriction \zeta\forces{\mathbb Q_{\omega_2}}{q_\xi(\zeta ) = (\check{h}_{\xi,\zeta},\check{w}_{\xi,\zeta})}.$

Let $\delta=\max\{\dom(q_0(0)),w_{0,\zeta_i}:i\le k\}$. Pick $\gamma<\omega_1$ so that each of the values 
$$\min(\dom(q_\gamma(0))\setminus \dom(q_0(0))),\min(w_{\gamma,\zeta_i}\setminus w_{0,\zeta_i}) \textrm{ for } i\le k,$$ are above $\delta$ (if defined).

Arguing as in  Claim~\ref{prevcl}, we see that $q_0$ and $q_\gamma$ are compatible conditions.
 \end{proof}
 This completes the proof of the Theorem.

\end{proof}

\begin{definition}
The symbol 
$$\k\to_p [\k]^2_{\lambda,<\mu}$$
for a partition $p:[\k]^2\to \theta$ means that for every coloring $f:[\k]^2\to \lambda$ there is a set $A\in [\k]^\k$ such that $|\ran (f\rest( [A]^2\cap p^{-1}(i)))|<\mu$  for all $i<\theta$. 
\end{definition}

Note that for $\mu\le \l$ this symbol is stronger than $\k\to_p [\k]^2_{\lambda\sm\lambda}$.
Thus the next theorem, which uses ideas from \cite{ShSt}, gives a stronger consistency than the previous one.

\begin{theorem}\label{medddasure}
Given any regular $\kappa>\aleph_1$ it is consistent that: 
\begin{itemize}
\item $\non(\mathcal L)= \aleph_1$
\item $\mathfrak b  = \aleph_2 = 2^{\aleph_0}$ 
\item $2^{\aleph_1} = \kappa$ 
\item there is a $p:[\omega_1]^2\to \omega$ such that $\omega_1\rightarrow_p [\omega_1]^2_{\omega,<\omega}$.
\end{itemize}
\end{theorem}
\begin{theorem}\label{medddcakjg66}
Given any regular $\kappa>\aleph_1$ it is consistent that: 
\begin{itemize}
\item $ \non(\mathcal M)= \aleph_1$
\item $\mathfrak b  = \aleph_1 $
\item $\mathfrak d  = \aleph_2 = 2^{\aleph_0}$
\item $2^{\aleph_1} = \kappa$ 
\item there is a $p:[\omega_1]^2\to \omega$ such that $\omega_1\rightarrow_p [\omega_1]^2_{\omega,<\omega}$.
\end{itemize}
\end{theorem}
The proofs of both theorems are similar, using ideas from \cite{ShSt};  only the proof of Theorem~\ref{medddasure} will be given in detail.
Both rely on the following definition:
\begin{definition}\label{defsei}
Let $\mu$ be some probability measure on $\omega$ under which each singleton has positive measure, for example $\mu(\{n\}) = 2^{-n}$.
A sequence of functions $\mathcal P = \{p_\eta\}_{\eta\in \omega_1}$ 
will be said to have full outer measure if:
\begin{itemize}
    \item $p_\eta:\eta \to \omega$
    \item for each $\eta\in \omega_1$ the set 
    $\{p_\beta\restriction \eta\}_{\beta> \eta}$ has measure one
in    the measure space $(\omega^\eta,\mu^\eta)$.
\end{itemize}
The sequence $\mathcal P$ is defined  to be nowhere meagre similarly, 
by requiring that for each $\eta\in \omega_1$ the set 
    $\{p_\beta\restriction \eta\}_{\beta> \eta}$ is nowhere meagre
in    $(\omega^\eta,\mu^\eta)$ with the usual product topology.
In both cases
define $p = p(\mathcal P)$ by $p(\alpha, \beta) = p_\beta(\alpha)$ if $\alpha < \beta$.
\end{definition}

By enumerating all functions from a countable ordinal into $\om$, we have:
\begin{proposition}\label{SierpoCH}
Assuming the Continuum Hypothesis there is a  sequence $\mathcal P = \{p_\eta\}_{\eta\in \omega_1}$ 
such that 
    $\{p_\beta\restriction \eta\}_{\beta> \eta} = \omega^\eta$ for each $\eta\in \omega_1$.
Hence $\mathcal P$ has 
 full outer measure  as in Definition~\ref{defsei}.
\end{proposition}
While it is, of course, impossible to preserve the property that $\{p_\beta\restriction \eta\}_{\beta> \eta} = \omega^\eta$ when adding reals, the goal of the following arguments is to show that the properties of Definition~\ref{defsei} can be preserved in certain circumstances.
The following definition is from \cite{ShSt} and will play a key role in this context.
\begin{definition}\label{Dkbkjb}
A function $\psi: \wfomom \to [\omega_1]^{<\aleph_0}$ satisfying that 
$\psi(s)\cap \psi(t)=\varnothing$ unless $s=t$ will be said to have {\em disjoint range}. 
If for each $t\in \wfomom$ there is $k$ such that $|\psi(t^\frown j)|<k$ for all $j\in \omega$
then $\psi$ will be called bounded with disjoint range.
If $G$ is a filter of subtrees of $ \wfomom$ 
and $\psi$ has disjoint range
define $$S(G,\psi) = \bigcup_{t\in\bigcap G} \psi(t) .$$
If $G$ is a generic filter of trees over a model $V$
define 
$$\mathcal S_b(G) = \SetOf{S(G,\psi)}{\psi\in V\text{ and $\psi$ is bounded with disjoint range}} $$
\end{definition}
It is shown in \cite{ShSt} that  Lemma~\ref{ShSt1} and Lemma~\ref{ShSt2} hold.
\begin{lemma}\label{ShSt1}
If $G\subseteq \mathbb L$ is generic over $V$ then
$\mathcal S_b(G)$ is a P-ideal in $V[G]$. 
\end{lemma}
Lemma~\ref{ThTod} is the content of \S3 of \cite{MR1441232}. Recall that if $\mathcal I$ is an ideal then $X$ is said to be orthogonal to $\mathcal I$ if $X\cap A$ is finite for each $A\in \mathcal I$.
\begin{lemma}
[Abraham and Todor{\v{c}}evi{\'c}]\label{ThTod}
Let $\mathcal I$ be a P-ideal on $\omega_1$ that is generated by a family of  $\aleph_1$ countable sets and such that
$\omega_1$ is not the union of countably many sets orthogonal to $\mathcal I$.
Then there is a proper partial order $\Poset_{\mathcal I}$, that adds no reals, even when iterated with countable support, such that there is a $\Poset_{\mathcal I}$-name  $\dot{Z}$ for an uncountable subset of $\omega_1$ such that 
$
1\forces{\Poset_{\mathcal I}}{(\forall \eta \in \omega_1) \ \Name{Z}\cap \eta\in \mathcal I} $.\end{lemma}

\begin{lemma}\label{ShSt2}
If $G\subseteq \mathbb L$ is generic over $V$ 
and $\Poset_{\mathcal S_b(G)}$ is the partial order of Lemma~\ref{ThTod} using Lemma~\ref{ShSt1}
and $H\subseteq \Poset_{\mathcal S_b(G)}$ is generic over $V[G]$ then in $V[G][H]$ there is an uncountable
$R\subseteq \omega_1$ such that $R\cap Y\neq \varnothing$ for each uncountable $Y\in V[G]$ and such that 
$[R]^{\aleph_0}\subseteq \mathcal S_b(G)$.
\end{lemma}
\begin{lemma}\label{Njkbkjw6jkp}
Let $\mathcal P$ be a sequence with full outer measure and suppose that $p = p(\mathcal P)$.
Suppose further that 
\begin{itemize}
\item $c:[\omega_1]^2 \to \omega$ 
\item $G\subseteq \mathbb L$ is generic over $V$
\item $H\subseteq \mathbb P_{\mathcal S_b(\dot{G})}$ is generic over $V[G]$.
\end{itemize}
Then there is an uncountable $X\subseteq \omega_1$ in $V[G][H]$ and
$L:\omega\to\omega$ such that 
$L(p(\alpha,\beta))>c(\alpha,\beta)$
for all $\{\alpha,\beta\}\in [X]^2$.
\end{lemma}
\begin{proof}
In $V[G]$ let  $L = \bigcap {G}$ be the Laver real. In
$V[G][H]$ let $R$ be the uncountable set given by Lemma~\ref{ShSt2}.
 Construct by induction distinct  $\rho_\xi\in R$ such that if $\eta \in \xi$ then    $L(p(\rho_\xi,\rho_\eta))>c(\rho_\xi,\rho_\eta) $.
To carry out the induction assume that $R_\eta= \{\rho_\xi\}_{\xi\in \eta}$ have been chosen and satisfy the inductive hypothesis.

By the choice of $R$ it follows that $R_\eta\in \mathcal S_b(G)$.
Since $\mathbb P_{\mathcal S_b(\dot{G})}$ adds no new reals it follows that
$R_\eta \in V[G]$ and so there
 is $T\in G$ and $\psi\in V$ with bounded, disjoint range such that
$T\forces{\mathbb L}{\dot{R}_\eta = S(\dot{G},\psi)}$. Let $\mu$ be so large that
$T\forces{\mathbb L}{\dot{R}_\eta \subseteq \mu}$ and let $r$ be the root of $T$.
For $t\in T$ define
$\mathcal W_t=\SetOf{x\in 2^\mu}{x\restriction  \psi(t)\text{ has constant value } |t|} 
$
and  then define
$$\mathcal W_t^+ = \SetOf{x\in 2^\mu}{(\exists^\infty s\in \textstyle{\suc_T(t)}) \ x\in \mathcal W_s } .$$
Note that $\mathcal W_t^+$  has  measure one in $2^\mu$ for each $t\supseteq r$. To see this note that
for a random $h\in 2^\mu$ the probability that $h(\zeta) = |t| + 1$ is $2^{-(|t|+1)}$.
Also, note that since $\psi$ is bounded --- see Definition~\ref{Dkbkjb} --- there is some $k$ such that
$|\psi(s)|\leq k$ for each $s\in \suc_T(t)$.
Hence, the probability of $h$
 belonging to $\mathcal W_{s}$ is bounded below by
 $2^{-(|t|+1)k}$ for all $s\in \suc_T(t)$ and these events are independent because the
 $\psi(s)$ are pairwise disjoint for $s\in \suc_T(t)$.

Define 
$f$ on $\bigcup_{j\leq |r|}\psi(r\restriction j)$ to have constant value $|r|$ and
note that the domain of $f$ is disjoint from
 each $\psi(s)$ where $s\supsetneq r$. Hence the probability that $f\subseteq h$ is non-zero and independent from belonging to each
 $\mathcal W^+$.
Since $p$ has full outer measure it follows that
$$\SetOf{\beta\in \omega_1}{f\subseteq p_\beta\restriction \mu\in\bigcap_{r\subseteq t\in T}\mathcal W_t^+}$$
is uncountable and belongs to $V[G]$.
Therefore by Lemma~\ref{ShSt2} there is some $\beta\in R\setminus R_\eta$ such that $f\subseteq p_\beta\restriction \mu$ and such that 
for all $t\in T$ containing $r$ there are infinitely many $s\in \textstyle{\suc_T(t)}$ such that 
$p(\alpha,\beta) = |s|$ for all $\alpha \in  \psi(s)$.

Using this and the definition of $f$, it is possible to start with $r$ and successively thin out the successors of each $t\in T$
to find a tree $T^*\subseteq  T$ with  root  $r$  such that 
$p(\alpha,\beta) = |t| $ for all $t\in T^*$ and for all $\alpha \in \psi(t)$.
Once again starting with $r$ and removing only finitely many elements of $\suc_{T^*}(t)$ for each $t\in T^*$ 
it is possible to find $T^{**}\subseteq T^*$ with root $r$ such that  
$$(\forall t\in T^{**})(\forall s\in \textstyle{\suc_{T^{**}}(t)})
(\forall \alpha \in  \psi(t)) \  s(|t|) = s(p(\alpha,\beta))>c(\alpha,\beta)$$
and this implies that 
$$T^{**}\forces{\mathbb L}{(\forall \alpha\in \dot{R}_\eta) \ \dot{L}(p(\alpha,\beta))>c(\alpha,\beta)} .$$
Since this holds for any $T$, genericity yields that in $V[G][H]$ there is some $\beta\in R\setminus R_\eta$ such that 
$L(p(\rho_\xi,\beta))>c(\rho_\xi,\beta)$ for each $ \xi \in \eta$. Define $\rho_\eta = \beta$ to continue the induction.
Since limit stages are immediate, this completes the proof.

\end{proof}

\begin{proof}[Proof of Theorem~\ref{medddasure}]
The required model is the one obtained by starting with a model of the Continuum Hypothesis in which $2^{\aleph_1}=\kappa$.
Then iterate with countable support the partial order
$\mathbb L\ast \Poset_{\mathcal S_b\dot{G})}$. In the initial model there is, by Proposition~\ref{SierpoCH},
a sequence with full outer measure.  To see this, begin by observing that it is shown in Theorem~7.3.39 of \cite{barjud}  that 
$\mathbb L$ preserves $ \sqsubseteq^{\bf Random}$. Since $\Poset_{\mathcal S(\dot{G})}$ is proper and adds no new reals it is immediate that it also preserves $ \sqsubseteq^{\bf Random}$. 
It follows by Theorem 6.1.13 of \cite{barjud} that the entire countable support iteration preserves outer measure sets and, hence, any sequence with full outer measure in the initial model maintains this property throughout the iteration.

To see that for every function $c:[\omega_1]^2\to \omega$ there is an uncountable set witnessing
 $\aleph_1\rightarrow_p[\aleph_1]_{\aleph_0,<\aleph_0}$
use Lemma~3.4 and Lemma~3.6 of \cite{MR1441232} to conclude that each partial order in the $\omega_2$ length iteration is proper and has the $\aleph_2$-pic of Definition~2.1 on page 409 of 
\cite{MR1623206}. By Lemma~2.4 on page 410 of 
\cite{MR1623206} it follows that the iteration has the $\aleph_2$ chain condition and, hence,
 that $c$ appears at some stage. It is then routine to apply Lemma~\ref{Njkbkjw6jkp}.

 That   $\mathfrak b=\aleph_2$ is a standard argument using that Laver forcing adds a dominating real. 
\end{proof}
 \begin{remark}
The proof of Theorem~\ref{medddcakjg66} is similar but uses Miller reals instead of Laver reals.
This requires that nowhere meagreness play the role of full outer measure. 
\end{remark}
\begin{remark}
Note that there is no partition $p$ such that $$\omega_1\rightarrow_p [\omega_1]^2_{\omega_1,<\omega_1}$$
because a colouring $c:[\omega_1]^2\to \omega_1$
that is a bijection will provide a counterexample.
\end{remark}


\section{Concluding Remarks and Open Questions}

It turns out, via Lemma \ref{implications}, that  getting strong coloring symbols over finite partitions is not harder than getting them without partititions; so one immediately  gets many strong coloring symbols over partitions outright in ZFC. If the  number of colors $\l$ raised to the number of cells in a partition is not too large, Lemma \ref{implications} applies again, and consequently all GCH symbols gotten by Erd\H os, Hajnal and Milner on $\k^+$ hold under the GCH over arbirary $\k$-partitions. Even without instances of the GCH, strong colorings symbols over countable partitions are valid in Cohen-type  forcing extenstions, by Theorems \ref{TheRecCohen} and \ref{TheRecSigmalinked}. 

Yet, it is not the case that every  time a strong-coloring symbol
holds at a successor of a regular, it also holds over countable partitions: by Theorem \ref{The8} and \ref{medddasure} the ZFC symbol
$\aleph_1\norrow[\aleph_1]^2_{\aleph_1}$, and hence all stronger ones, consistently fail quite badly over  sufficiently generic countable partitions.
Thus, strong coloring symbols over partitions are a subject of their own, in which the independence phenomenon is manifested  prominently.

Many natural questions about the combinatorial and set-theoretic
connections between coloring and partition arise. We hope that 
this subject will get attention in the near future both in the infinite combinatorics and in the forcing communities. For  example,
by Fact \ref{nosingle}, there is always a set of $2$-partitions of $[\k^+]^2$ such that no coloring is strong over all of them. What is the least cardinality of such a set? In the case of $\theta=\kappa=\aleph_0$,
the results in Section \ref{ind} show that this cardinal may be as small as 1 or at least as large as $\aleph_2=\k^{++}$.
Can this number ever be $\k$ or, say,  $\k^+<2^\k$? 

We conclude with a short selection of open  questions.  

\begin{question}
If $\Pr_1(\aleph_1,\aleph_1,\aleph_1,\aleph_0)_p$ holds for all countable $p$, does also  $\Pr_0(\aleph_1,\aleph_1,\aleph_1,\aleph_0)_p$ hold for all countable $p$? 
\end{question}
  \begin{question}
  Suppose   $\Pr_0(\aleph_1,\aleph_1,\aleph_0,\aleph_0)_p$ holds for some countable partition $p$. Does 
$\Pr_0(\aleph_1,\aleph_1,\aleph_1,\aleph_0)_p$ hold as well? 
  \end{question}
Without partitions, both implications above hold.


\begin{question}
Does $MA_\text{$\sigma$-linked} $ or $\mathfrak p = \mathfrak c$ or even full $MA_{\aleph_1}$ imply that $\Pr_{0}(\aleph_1,\aleph_1,\aleph_1,\aleph_0)_{\bar p}$ holds  
for every $\omega_1$ sequence of partitions $\bar{p} = \lng p_\d:\d<\omega_1\rng$ such that 
$p_\d:[\omega_1]^2\to \omega$?
\end{question}

\begin{question}
Is it consistent that there is a partition $p$ such that $$\aleph_1\rightarrow_p [\aleph_1]^2_{\aleph_0,<k}$$
for some integer $k$?
\end{question}

\begin{question}Is 
\[\aleph_2\norrow_p[\faktor{{\scriptstyle{{\aleph_0}\circledast\aleph_2}}}{{}^{1\circledast\aleph_2}}]^2_{\aleph_2}\]
consistent for all $\aleph_0$- or $\aleph_1$-partitions $p$?
That is, can there be a coloring $f:[\om_2]^2\to \om_2$
such that for every (one, or sequence of $\omega_2$ many) $\om_1$-partition(s) of 
$[\om_2]^2$, for every $B\in [\om_2]^{\om_2}$, for
all but \emph{finitely} many $\a<\om_2$ there is $i<\om_1$ such that for 
every color $\zeta<\om_2$ there is $\b\in B$ such that $p(\a,\b)=i$ and $f(\a,\b)=\zeta$. 
\end{question}

The consistency of this symbol  is open even without the $p$. A negative
answer may be easier to get with  $p$. 

\bigskip
{\sl Added in proof}: Problems 46--49 above are solved  in  \cite{KRS}.

\end{document}